\documentclass[12pt,english,magyar,british]{article}
\usepackage[T1]{fontenc}
\usepackage[latin2,latin9]{inputenc}
\usepackage{geometry}
\usepackage{amsmath}
\usepackage{amsthm}
\usepackage{amssymb}

\makeatletter
\theoremstyle{plain}
\newtheorem{thm}{\protect\theoremname}[section]
  \theoremstyle{remark}
  \newtheorem{rem}[thm]{\protect\remarkname}
  \theoremstyle{plain}
  \newtheorem{cor}[thm]{\protect\corollaryname}
  \theoremstyle{plain}
  \newtheorem{conjecture}[thm]{\protect\conjecturename}
  \theoremstyle{plain}
  \newtheorem{prop}[thm]{\protect\propositionname}
  \theoremstyle{plain}
  \newtheorem{lem}[thm]{\protect\lemmaname}
  \theoremstyle{remark}
  \newtheorem{note}[thm]{\protect\notename}

\makeatother

\usepackage{babel}
  \addto\captionsbritish{\renewcommand{\conjecturename}{Conjecture}}
  \addto\captionsbritish{\renewcommand{\corollaryname}{Corollary}}
  \addto\captionsbritish{\renewcommand{\lemmaname}{Lemma}}
  \addto\captionsbritish{\renewcommand{\notename}{Note}}
  \addto\captionsbritish{\renewcommand{\propositionname}{Proposition}}
  \addto\captionsbritish{\renewcommand{\remarkname}{Remark}}
  \addto\captionsbritish{\renewcommand{\theoremname}{Theorem}}
  \addto\captionsenglish{\renewcommand{\conjecturename}{Conjecture}}
  \addto\captionsenglish{\renewcommand{\corollaryname}{Corollary}}
  \addto\captionsenglish{\renewcommand{\lemmaname}{Lemma}}
  \addto\captionsenglish{\renewcommand{\notename}{Note}}
  \addto\captionsenglish{\renewcommand{\propositionname}{Proposition}}
  \addto\captionsenglish{\renewcommand{\remarkname}{Remark}}
  \addto\captionsenglish{\renewcommand{\theoremname}{Theorem}}
  \addto\captionsmagyar{\renewcommand{\conjecturename}{Feltevés}}
  \addto\captionsmagyar{\renewcommand{\corollaryname}{Következmény}}
  \addto\captionsmagyar{\renewcommand{\lemmaname}{Segédtétel}}
  \addto\captionsmagyar{\renewcommand{\notename}{Megjegyzés}}
  \addto\captionsmagyar{\renewcommand{\propositionname}{Állítás}}
  \addto\captionsmagyar{\renewcommand{\remarkname}{Észrevétel}}
  \addto\captionsmagyar{\renewcommand{\theoremname}{Tétel}}
  \providecommand{\conjecturename}{Conjecture}
  \providecommand{\corollaryname}{Corollary}
  \providecommand{\lemmaname}{Lemma}
  \providecommand{\notename}{Note}
  \providecommand{\propositionname}{Proposition}
  \providecommand{\remarkname}{Remark}
\providecommand{\theoremname}{Theorem}

\begin{document}
\selectlanguage{english}%
\date{}

\title{Dimension approximation of attractors of graph directed IFSs by self-similar
sets}

\author{\'Abel Farkas}
\maketitle
\selectlanguage{british}%
\begin{abstract}
We show that for the attractor $(K_{1},\dots,K_{q})$ of a graph directed
iterated function system, for each $1\leq j\leq q$ and $\varepsilon>0$
there exists a self-similar set $K\subseteq K_{j}$ that satisfies
the strong separation condition and $\dim_{H}K_{j}-\varepsilon<\dim_{H}K$.
We show that we can further assume convenient conditions on the orthogonal
parts and similarity ratios of the defining similarities of $K$.
Using this property as a `black box' we obtain results on a range
of topics including on dimensions of projections, intersections, distance
sets and sums and products of sets.
\end{abstract}

\section{Introduction}

The main goal of this paper is to develop a tool to deduce results
about attractors of graph directed iterated function systems from
results that are known for self-similar sets. We proceed by finding
a self-similar subset of a given graph directed attractor such that
the Hausdorff dimension of the self-similar set is arbitrary close
to that of the graph directed attractor and the self-similar set has
convenient properties such as the strong separation condition. Similar
methods, approximating self-similar sets with well-behaved self-similar
sets in the plane, were used by Peres and Shmerkin \cite[Proposition 6, Theorem 2]{Peres-Schmerkin Resonance between Cantor sets},
by Orponen \cite[Lemma 3.4]{Orponen distance set} and in higher dimensions
by Farkas \cite[Proposition 1.8]{Farkas-Linim-only}. Fraser and Pollicott
\cite[Proposition 2.5]{Fraser-Pollicott SSOFT} showed a result of
similar nature for conformal systems on subshifts of finite type.
After stating the approximation theorems we deduce many corollaries
relating to the dimension of projections and smooth images, the distance
set conjecture, the dimension of arithmetic products and dimension
conservation.

The paper is organised as follows. In Section \ref{subsec:Notations}
we introduce notation and definitions. In Section \ref{subsec:Main results}
we state the main results of this paper, the approximation theorems.
In Section \ref{subsec:Application GDA} we deduce various corollaries
that follow from the approximation theorems. In Section \ref{subsec:SSOFT section}
we restate the approximation theorems for subshifts of finite type
and note that the results of Section \ref{subsec:Application GDA}
remain true. Section \ref{sec:Proofs} contains the proofs of the
theorems.
\selectlanguage{english}%

\subsection{Definitions and Notations\label{subsec:Notations}}

A \textit{self-similar iterated function system} (SS-IFS) in $\mathbb{R}^{d}$
is a finite collection of maps $\left\{ S_{i}\right\} _{i=1}^{m}$
from $\mathbb{R}^{d}$ to $\mathbb{R}^{d}$ such that all the $S_{i}$
are contracting similarities. The \textit{attractor} of the SS-IFS
is the unique nonempty compact set $K\subseteq\mathbb{R}^{d}$ such
that $K=\bigcup_{i=1}^{m}S_{i}(K)$. The attractor of an SS-IFS is
called a \textit{self-similar set}. We say that the SS-IFS $\left\{ S_{i}\right\} _{i=1}^{m}$
satisfies the \textit{strong separation condition} (SSC) if the sets
$\left\{ S_{i}(K)\right\} _{i=1}^{m}$ are disjoint. 

Let $\left\{ S_{i}\right\} _{i=1}^{m}$ be an SS-IFS. Then every $S_{i}$
can be uniquely decomposed as
\begin{equation}
S_{i}(x)=r_{i}T_{i}(x)+v_{i}\label{eq:S_ideconposition}
\end{equation}
for all $x\in\mathbb{R}^{d}$, where $0<r_{i}<1$, $T_{i}$ is an
orthogonal transformation and $v_{i}\in\mathbb{R}^{d}$ is a translation,
for all indices $i$. The unique solution $s$ of the equation
\begin{equation}
\sum_{i=1}^{m}r_{i}^{s}=1\label{eq: S-dim sum}
\end{equation}
is called the \textit{similarity dimension} of the SS-IFS. It is well-known
that if an SS-IFS satisfies the SSC then $0<\mathcal{H}^{s}(K)<\infty$.
Let $\mathcal{T}$ denote the group generated by the orthogonal transformations
$\left\{ T_{i}\right\} _{i=1}^{m}$. We call $\mathcal{T}$ the \textit{transformation
group} of the SS-IFS.

We denote the set $\left\{ 1,2,\ldots,m\right\} $ by $\mathcal{I}$.
Let $\boldsymbol{\mathbf{i}}=(i_{1},\ldots,i_{k})\in\mathcal{I}^{k}$
i.e. $\boldsymbol{\mathbf{i}}$ is a $k$-tuple of indices. Then we
write $S_{\boldsymbol{\mathbf{i}}}=S_{i_{1}}\circ\ldots\circ S_{i_{k}}$
and $K_{\boldsymbol{\mathbf{i}}}=S_{\boldsymbol{\mathbf{i}}}(K)$.
When the similarities are decomposed as in (\ref{eq:S_ideconposition})
we write $r_{\boldsymbol{\mathbf{i}}}=r_{i_{1}}\cdot\ldots\cdot r_{i_{k}}$
and $T_{\boldsymbol{\mathbf{i}}}=T_{i_{1}}\circ\ldots\circ T_{i_{k}}$.
For an overview of the theory of self-similar sets see, for example,
\cite{Falconer1,Falconer Techniques,Hutchinson,Mattilakonyv,Schief OSC}.

Let $G\left(\mathcal{V},\mathcal{E}\right)$ be a directed graph,
where $\mathcal{V}=\left\{ 1,2,\ldots,q\right\} $ is the set of vertices
and $\mathcal{E}$ is the finite set of directed edges such that for
each $i\in\mathcal{V}$ there exists at least one $e\in\mathcal{E}$
starting from $i$. Let $\mathcal{E}_{i,j}$ denote the set of edges
from vertex $i$ to vertex $j$ and $\mathcal{E}_{i,j}^{k}$ denote
the set of sequences of $k$ edges $\left(e_{1},\ldots,e_{k}\right)$
which form a directed path from vertex $i$ to vertex $j$. A \textit{graph
directed iterated function system} (GD-IFS) in $\mathbb{R}^{d}$ is
a finite collection of maps $\left\{ S_{e}:e\in\mathcal{E}\right\} $
from $\mathbb{R}^{d}$ to $\mathbb{R}^{d}$ such that all the $S_{e}$
are contracting similarities. The \textit{attractor} of the GD-IFS
is the unique $q$-tuple of nonempty compact sets $\left(K_{1},\ldots,K_{q}\right)$
such that
\begin{equation}
K_{i}=\bigcup_{j=1}^{q}\bigcup_{e\in\mathcal{E}_{i,j}}S_{e}(K_{j}).\label{eq:GDAdef}
\end{equation}
The attractor of a GD-IFS is called a \textit{graph directed attractor}.

Let $\left\{ S_{e}:e\in\mathcal{E}\right\} $ be a GD-IFS. Then every
$S_{e}$ can be uniquely decomposed as
\begin{equation}
S_{e}(x)=r_{e}T_{e}(x)+v_{e}\label{eq:S_edeconposition for GDA}
\end{equation}
for all $x\in\mathbb{R}^{d}$, where $0<r_{e}<1$, $T_{e}$ is an
orthogonal transformation and $v_{e}\in\mathbb{R}^{d}$ is a translation,
for all edges $e$.

Let $\boldsymbol{\mathbf{e}}=(e_{1},\ldots,e_{k})\in\mathcal{E}_{i,j}^{k}$,
then we write $S_{\boldsymbol{\mathbf{e}}}$ for $S_{e_{1}}\circ\ldots\circ S_{e_{k}}$
and $K_{\boldsymbol{\mathbf{e}}}$ for $S_{\boldsymbol{\mathbf{e}}}(K_{j})\subseteq K_{i}$.
If the similarities are decomposed as in (\ref{eq:S_edeconposition for GDA})
then we write $r_{\boldsymbol{\mathbf{e}}}$ for $r_{e_{1}}\cdot\ldots\cdot r_{e_{k}}$
and $T_{\boldsymbol{\mathbf{e}}}$ for $T_{e_{1}}\circ\ldots\circ T_{e_{k}}$.
If $\boldsymbol{\mathbf{e}}=(e_{1},\ldots,e_{k})\in\mathcal{E}_{i,j}^{k}$
and $\boldsymbol{\mathbf{f}}=(f_{1},\ldots,f_{n})\in\mathcal{E}_{j,l}^{n}$
for $i,j,l\in\mathcal{V}$ then we write $\boldsymbol{\mathbf{e}}*\boldsymbol{\mathbf{f}}$
for $(e_{1},\ldots,e_{k},f_{1},\ldots,f_{n})\in\mathcal{E}_{i,l}^{k+n}$.

The directed graph $G\left(\mathcal{V},\mathcal{E}\right)$ is called
\textsl{strongly connected} if for every pair of vertices $i$ and
$j$ there exist a directed path from $i$ to $j$ and a directed
path from $j$ to $i$. We say that the GD-IFS $\left\{ S_{e}:e\in\mathcal{E}\right\} $
is \textsl{strongly connected} if $G\left(\mathcal{V},\mathcal{E}\right)$
is strongly connected. For an overview of the theory of graph directed
attractors see, for example, \cite{Falconer Techniques,Mauldin Williams construction,Wang GDA OSC}.

The set $\mathcal{C}_{i}:=\bigcup_{k=1}^{\infty}\mathcal{E}_{i,i}^{k}$
is the set of directed cycles of $G\left(\mathcal{V},\mathcal{E}\right)$
that start and end in vertex $i$. Equipped with the $*$ operation
$\mathcal{C}_{i}$ becomes a semigroup. Let $\mathcal{T}_{i,G}$ denote
the group generated by the transformations
\[
\left\{ T_{e_{1}}\circ\ldots\circ T_{e_{k}}:\left(e_{1},\ldots,e_{k}\right)\in\mathcal{C}_{i}\right\} 
\]
and we call $\mathcal{T}_{i,G}$ the \textit{$i$-th transformation
group} of the GD-IFS. It is easy to see that if the GD-IFS is strongly
connected then $\mathcal{T}_{i,G}$ is conjugate to $\mathcal{T}_{j,G}$
for all $i,j\in\mathcal{V}$ and hence
\[
\left|\mathcal{T}_{i,G}\right|=\left|\mathcal{T}_{j,G}\right|,
\]
where $\left|.\right|$ denote the cardinality of a set.

To avoid the singular non-interesting case, when every $K_{i}$ is
a single point we make the global assumption throughout the whole
paper that there exists $i\in\mathcal{V}$ such that $K_{i}$ contains
at least two points. This implies that if $\left\{ S_{e}:e\in\mathcal{E}\right\} $
is strongly connected then every $K_{j}$ contains at least two points.
Note that if $\left\{ S_{e}:e\in\mathcal{E}\right\} $ is strongly
connected then this assumption also implies that $\dim_{H}K_{i}=\dim_{H}K_{j}>0$
for all $i,j\in\mathcal{V}$ even with no separation condition.

\selectlanguage{british}%
We use the following notation throughout the paper. We let $\dim_{H}K$
be the Hausdorff dimension of a set $K$. Let $\mathbb{O}_{d}$ be
the group of all orthogonal transformations on $\mathbb{R}^{d}$ and
$\mathbb{SO}_{d}$ be the group of special orthogonal transformations
on $\mathbb{R}^{d}$ both of them equipped with the usual topology.
If $\mathcal{T}$ is a set of orthogonal transformations then $\overline{\mathcal{T}}$
denotes the closure of $\mathcal{T}$. We denote the identity map
of $\mathbb{R}^{d}$ by $Id_{\mathbb{R}^{d}}$. For a linear transformation
$L:\mathbb{R}^{d}\longrightarrow\mathbb{R}^{d_{2}}$\foreignlanguage{english}{
the Euclidean operator norm is} \foreignlanguage{english}{
\[
\left\Vert T\right\Vert =\sup_{x\in\mathbb{R}^{d},\left\Vert x\right\Vert =1}\left\Vert Tx\right\Vert ,
\]
where $\left\Vert y\right\Vert $ denotes the Euclidean norm of $y\in\mathbb{R}^{d}$.}
Let $\mathcal{L}^{d}$ be the $d$-dimensional Lebesgue measure. We
write $B(x,\gamma)\subseteq\mathbb{R}^{d}$ for the open ball of radius
$\gamma$ centered at $x$.
\selectlanguage{english}%

\subsection{The approximation theorems\label{subsec:Main results}}

\selectlanguage{british}%
In this section we state the main results of this paper, the approximation
theorems. Given a graph directed attractor we find a self-similar
subset that has arbitrarily close dimension and satisfies the strong
separation condition. The first result shows that we can further require
that the transformation group of the self-similar set is dense in
that of the graph directed attractor.
\begin{thm}
\label{thm: MAIN}Let $\left\{ S_{e}:e\in\mathcal{E}\right\} $ be
a strongly connected GD-IFS in $\mathbb{R}^{d}$ with attractor $\left(K_{1},\ldots,K_{q}\right)$
and let $j\in\mathcal{V}$. Then for every $\varepsilon>0$ there
exists an SS-IFS $\left\{ S_{i}\right\} _{i=1}^{m}$ that satisfies
the SSC, with attractor $K$ such that $K\subseteq K_{j}$, $\dim_{H}K_{j}-\varepsilon<\dim_{H}K$
and the transformation group $\mathcal{T}$ of $\left\{ S_{i}\right\} _{i=1}^{m}$
is dense in $\mathcal{T}_{j,G}$.
\end{thm}

\begin{rem}
\label{quocient remark}Let $\left\{ S_{e}:e\in\mathcal{E}\right\} $
and $\left\{ \widehat{S}_{e}:e\in\widehat{\mathcal{E}}\right\} $
be two strongly connected GD-IFS in $\mathbb{R}^{d}$, with attractors
$\left(K_{1},\ldots,K_{q}\right)$ and $\left(\widehat{K}_{1},\ldots,\widehat{K}_{\widehat{q}}\right)$,
such that there exist $j\in\mathcal{V}$, $\widehat{j}\in\widehat{\mathcal{V}}$
and $\boldsymbol{\mathbf{e}}\in\mathcal{C}_{j}$, $\boldsymbol{\mathbf{f}}\in\widehat{\mathcal{C}}_{\widehat{j}}$
such that $\log r_{\boldsymbol{\mathbf{e}}}/\log r_{\boldsymbol{\mathbf{f}}}\notin\mathbb{Q}$.
Then we can find SS-IFS $\left\{ S_{i}\right\} _{i=1}^{m}$ and $\left\{ \widehat{S}_{i}\right\} _{i=1}^{\widehat{m}}$
that satisfy the SSC, with attractors $K$ and $\widehat{K}$ such
that $K\subseteq K_{j}$, $\dim_{H}K_{j}-\varepsilon<\dim_{H}K$,
$\widehat{K}\subseteq\widehat{K}_{j}$, $\dim_{H}\widehat{K}_{j}-\varepsilon<\dim_{H}\widehat{K}$,
the transformation group $\mathcal{T}$ of $\left\{ S_{i}\right\} _{i=1}^{m}$
is dense in $\mathcal{T}_{j,G}$, the transformation group $\widehat{\mathcal{T}}$
of $\left\{ \widehat{S}_{i}\right\} _{i=1}^{\widehat{m}}$ is dense
in $\widehat{\mathcal{T}}_{j,G}$ and $\log r_{1}/\log\widehat{r}_{1}\notin\mathbb{Q}$.
See Remark \ref{log r}.
\end{rem}

In the next result instead of the dense subgroup condition we can
require that the first level cylinder sets of the self-similar set
are the same size and `roughly homothetic', i.e. all the similarity
ratios are the same and the orthogonal parts are $\varepsilon$-close.

\begin{thm}
\label{thm:approx trafo thm}Let $\left\{ S_{e}:e\in\mathcal{E}\right\} $
be a strongly connected GD-IFS in $\mathbb{R}^{d}$ with attractor
$\left(K_{1},\ldots,K_{q}\right)$ and let $j\in\mathcal{V}$. Then
for every $\varepsilon>0$ and $O\in\overline{\mathcal{T}_{j,G}}$
there exist $r\in(0,1)$ and an SS-IFS $\left\{ S_{i}\right\} _{i=1}^{m}$
that satisfies the SSC, with attractor $K$ such that $K\subseteq K_{j}$,
$\dim_{H}K_{j}-\varepsilon<\dim_{H}K$, and $\left\Vert T_{i}-O\right\Vert <\varepsilon$,
$r_{i}=r$ for every $i\in\left\{ 1,\dots,m\right\} $.
\end{thm}

If the transformation group is finite then we can even get that $T_{i}=O$
for every $i$.

\begin{cor}
\label{cor:finite approx cor}Let $\left\{ S_{e}:e\in\mathcal{E}\right\} $
be a strongly connected GD-IFS in $\mathbb{R}^{d}$ with attractor
$\left(K_{1},\ldots,K_{q}\right)$, let $j\in\mathcal{V}$ and assume
that $\mathcal{T}_{j,G}$ is a finite group. Then for every $\varepsilon>0$
and $O\in\mathcal{T}_{j,G}$ there exist $r\in(0,1)$ and an SS-IFS
$\left\{ S_{i}\right\} _{i=1}^{m}$ that satisfies the SSC, with attractor
$K$ such that $K\subseteq K_{j}$, $\dim_{H}K_{j}-\varepsilon<\dim_{H}K$,
and $T_{i}=O$, $r_{i}=r$ for every $i\in\left\{ 1,\dots,m\right\} $.
\end{cor}

Corollary \ref{cor:finite approx cor} follows easily from Theorem
\ref{thm:approx trafo thm} by letting
\[
\varepsilon=\min\left\{ \left\Vert T-O\right\Vert :T,O\in\mathcal{T}_{j,G},T\neq O\right\} >0.
\]

One cannot hope to have in Theorem \ref{thm:approx trafo thm} that
$T_{i}=O$ for every $i$ because in $\mathbb{R}^{3}$ there exist
two rotations around lines that generate a free group over two elements
and so $T_{\boldsymbol{\mathbf{i}}}$ might all be different for every
finite word $\boldsymbol{\mathbf{i}}$. However, rotations on the
plane commute, hence we can get all $T_{i}$ to be the same.

\begin{thm}
\label{thm:planar dense thm}Let $\left\{ S_{e}:e\in\mathcal{E}\right\} $
be a strongly connected GD-IFS in $\mathbb{R}^{2}$ with attractor
$\left(K_{1},\ldots,K_{q}\right)$ and let $j\in\mathcal{V}$. Then
for every $\varepsilon>0$ there exist $r\in(0,1)$, an orthogonal
transformation $O\in\mathcal{T}_{j,G}\cap\mathbb{SO}_{2}$ and an
SS-IFS $\left\{ S_{i}\right\} _{i=1}^{m}$ that satisfies the SSC,
with attractor $K$ such that $K\subseteq K_{j}$, $\dim_{H}K_{j}-\varepsilon<\dim_{H}K$,
the transformation group $\mathcal{T}$ of $\left\{ S_{i}\right\} _{i=1}^{m}$
is dense in $\mathcal{T}_{j,G}\cap\mathbb{SO}_{2}$ and $T_{i}=O$,
$r_{i}=r$ for every $i\in\left\{ 1,\dots,m\right\} $.
\end{thm}

\subsection{Application of the approximation theorems\label{subsec:Application GDA}}

In this section we give applications of the dimension approximation
results. The first application generalises a result of Hochman and
Shmerkin \cite[Corollary 1.7]{Hochman-Schmerkin-local entropy} on
self-similar sets with SSC to graph directed attractors with no separation
condition.

\selectlanguage{english}%
Let $0<l\leq d$ be integers and let $G_{d,l}$ denote the \textit{Grassmann
manifold} of $l$-dimensional linear subspaces of $\mathbb{R}^{d}$
equipped with the usual topology (see for example \cite[Section 3.9]{Mattilakonyv}).

\selectlanguage{british}%

\begin{thm}
\label{thm:hoch-shmer}Let $\left\{ S_{e}:e\in\mathcal{E}\right\} $
be a strongly connected GD-IFS in $\mathbb{R}^{d}$ with attractor
$\left(K_{1},\ldots,K_{q}\right)$, let $j\in\mathcal{V}$\foreignlanguage{english}{,
let $U$ be an open neighbourhood of $K_{j}$ and assume that there
exists $M\in G_{d,l}$ such that the set $\left\{ O(M):O\in\mathcal{T}_{j,G}\right\} $
is dense in $G_{d,l}$ for some $1\leq l<d$. Then $\dim_{H}\left(g(K_{j})\right)=\min\left\{ \dim_{H}(K_{j}),l\right\} $
for every continuously differentiable map $g:U\longrightarrow\mathbb{R}^{l}$
such that $\mathrm{rank}(g'(x))=l$ for some $x\in K_{j}$.}
\end{thm}

Since \foreignlanguage{english}{$\mathrm{rank}(g'(x))=l$ it follows
that there exists an open neighbourhood $V$ of $x$ such that $\mathrm{rank}(g'(y))=l$
for every $y\in V$.} Because $g$ is a Lipschitz map $\dim_{H}\left(g(K_{j})\right)\leq\min\left\{ \dim_{H}(K_{j}),l\right\} $
is straightforward. Taking a small cylinder set inside $V$ we can
further assume that $K_{j}\subseteq V$ (see Lemma \ref{lem:dense piece lem}).
The opposite inequality follows by finding $K\subseteq K_{j}$ as
in Theorem \ref{thm: MAIN}. Applying \cite[Corollary 1.7]{Hochman-Schmerkin-local entropy}
to $K$ finishes the proof of Theorem \ref{thm:hoch-shmer}.

The following corollary applies to $g:U\longrightarrow\mathbb{R}^{d_{2}}$
where the dimension $d_{2}$ of the ambient space of the image can
be greater than $l$.

\begin{cor}
\label{cor:hoch-shmer cor}Let $\left\{ S_{e}:e\in\mathcal{E}\right\} $
be a strongly connected GD-IFS in $\mathbb{R}^{d}$ with attractor
$\left(K_{1},\ldots,K_{q}\right)$, let $j\in\mathcal{V}$\foreignlanguage{english}{,
let $U$ be an open neighbourhood of $K_{j}$ and assume that there
exists $M\in G_{d,l}$ such that the set $\left\{ O(M):O\in\mathcal{T}_{j,G}\right\} $
is dense in $G_{d,l}$ for some $1\leq l<d$. If $g:U\longrightarrow\mathbb{R}^{d_{2}}$
is a continuously differentiable map ($d_{2}\in\mathbb{N}$) such
that $\mathrm{rank}(g'(x))=l$ for every $x\in K_{j}$ and either
of the following conditions is satisfied}

\selectlanguage{english}%
(i) $g\in C^{\infty}$,

(ii) $\dim_{H}(K_{j})\leq l$,

\noindent then $\dim_{H}\left(g(K_{j})\right)=\min\left\{ \dim_{H}(K_{j}),l\right\} $.
\end{cor}

Corollary \ref{cor:hoch-shmer cor} can be deduced from Theorem \ref{thm:hoch-shmer}
as \cite[Corollary 1.7]{Farkas-Linim-only} is deduced from \cite[Theorem 1.6]{Farkas-Linim-only}.

Another well-studied topic is Falconer's distance set conjecture.
For a set $K$ we denote the \textsl{distance set of $K$} by $D(K)=\left\{ \left\Vert x-y\right\Vert :x,y\in K\right\} $.
The conjecture is the following (see \cite{Falconer distance set}):

\begin{conjecture}
Let $K\subseteq\mathbb{R}^{d}$ be an analytic set. If $\dim_{H}K\geq\frac{d}{2}$
then $\dim_{H}D(K)=1$, if $\dim_{H}K>\frac{d}{2}$ then $\mathcal{L}^{1}\left(D(K)\right)>0$.
\end{conjecture}

Orponen \cite[Theorem 1.2]{Orponen distance set} showed that for
a planar self-similar set $K$ if $\mathcal{H}^{1}\left(K\right)>0$
then $\dim_{H}D(K)=1$. B\'ar\'any \cite[Corollary 2.8]{Balazs Barany R^3}
extended this result by showing that if $K$ is a self-similar set
in $\mathbb{R}^{2}$ and $\dim_{H}K\geq1$ then $\dim_{H}D(K)=1$.
B\'ar\'any \cite[Theorem 1.2]{Balazs Barany R^3} also showed that
if $K\subseteq\mathbb{R}^{3}$, every $T_{i}=Id_{\mathbb{R}^{3}}$
in the SS-IFS of $K$ and $\dim_{H}K>1$ then $\dim_{H}D(K)=1$. Using
our approximation theorems we deduce these results for graph directed
attractors.

We define the \textsl{pinned distance set of $K\subseteq\mathbb{R}^{d}$}
to be $D_{x}(K)=\left\{ \left\Vert x-y\right\Vert :y\in K\right\} $
for the pin $x\in\mathbb{R}^{d}$. Clearly $\dim_{H}D(K)\geq\dim_{H}D_{x}(K)$
for every set $K\subseteq\mathbb{R}^{d}$ with $x\in K$. For a fixed
$x\in\mathbb{R}^{d}$ the map $D_{x}(y)=\left\Vert x-y\right\Vert $
is a locally Lipschitz map. Hence $\dim_{H}D_{x}(K)\leq\min\left\{ \dim_{H}(K_{j}),1\right\} $
for every set $K\subseteq\mathbb{R}^{d}$.

\begin{thm}
\label{thm:distance set R2}Let $\left\{ S_{e}:e\in\mathcal{E}\right\} $
be a strongly connected GD-IFS in $\mathbb{R}^{2}$ with attractor
$\left(K_{1},\ldots,K_{q}\right)$ and let $j\in\mathcal{V}$. Then
$\dim_{H}D_{x}(K_{j})=\min\left\{ \dim_{H}(K_{j}),1\right\} $ for
every $x\in\mathbb{R}^{2}$. In particular, if $\dim_{H}K_{j}\geq1$
then $\dim_{H}D(K_{j})=\dim_{H}D_{x}(K_{j})=1$ for every $x\in K_{j}$.
\end{thm}

\begin{proof}
Let $D_{x}(y)=\left\Vert x-y\right\Vert $ for $x,y\in\mathbb{R}^{2}$.
We can find $K\subseteq K_{j}$ as in Theorem \ref{thm: MAIN} for
every $\varepsilon>0$. Let $K_{\boldsymbol{\mathbf{i}}}$ be a cylinder
set such that $x\notin K_{\boldsymbol{\mathbf{i}}}$. Then $\Lambda=K_{\boldsymbol{\mathbf{i}}}$
is a self-similar set and $g(y)=D_{x}(y)$ satisfies the conditions
of \cite[Theorem 2.7]{Balazs Barany R^3} hence
\[
\dim_{H}D(K_{j})\geq\dim_{H}D_{x}(K)\geq\dim_{H}D_{x}(K_{\boldsymbol{\mathbf{i}}})=\min\left\{ \dim_{H}(K_{\boldsymbol{\mathbf{i}}}),1\right\} =\min\left\{ \dim_{H}(K),1\right\} 
\]
and this completes the proof.
\end{proof}

\begin{thm}
Let $\left\{ S_{e}:e\in\mathcal{E}\right\} $ be a strongly connected
GD-IFS in $\mathbb{R}^{3}$ with attractor $\left(K_{1},\ldots,K_{q}\right)$
and let $j\in\mathcal{V}$. If $\left|\mathcal{T}_{j,G}\right|<\infty$
and $\dim_{H}K_{j}>1$ then $\dim_{H}D(K_{j})=1$.
\end{thm}

By Corollary \ref{cor:finite approx cor} we can find an SS-IFS $\left\{ S_{i}\right\} _{i=1}^{m}$
with $K\subseteq K_{j}$ such that $\dim_{H}K>1$ and every $T_{i}=Id_{\mathbb{R}^{3}}$.
Then the statement follows by applying \cite[Theorem 1.2]{Balazs Barany R^3}
for $K$.

\begin{thm}
\label{thm:Rd D(K)}Let $\left\{ S_{e}:e\in\mathcal{E}\right\} $
be a strongly connected GD-IFS in $\mathbb{R}^{d}$ with attractor
$\left(K_{1},\ldots,K_{q}\right)$ and let $j\in\mathcal{V}$. If
\foreignlanguage{english}{there exists $M\in G_{d,1}$ such that the
set $\left\{ O(M):O\in\mathcal{T}_{j,G}\right\} $ is dense in $G_{d,1}$}
then $\dim_{H}D_{x}(K_{j})=\min\left\{ \dim_{H}(K_{j}),1\right\} $.
In particular, if $\dim_{H}K_{j}\geq1$ then $\dim_{H}D(K_{j})=\dim_{H}D_{x}(K_{j})=1$.
\end{thm}

Theorem \ref{thm:Rd D(K)} follows by applying Theorem \ref{thm:hoch-shmer}
to $g(y)=D_{x}(y)=\left\Vert x-y\right\Vert $ for some arbitrarily
chosen $x\in K_{j}$.

B\'ar\'any's paper \cite{Balazs Barany R^3} provides information
about the dimension of the arithmetic products of self-similar sets
in the line. In \cite[Corollary 2.9]{Balazs Barany R^3} he shows
that if $K\subseteq\mathbb{R}$ is a self-similar set then $\dim_{H}\left(K\cdot K\right)=\min\left\{ 2\dim_{H}(K),1\right\} $
where $A\cdot B=\left\{ a\cdot b:a\in A,b\in B\right\} $ for two
sets $A$ and $B$. In particular, if $\dim_{H}K\geq\frac{1}{2}$
then $\dim_{H}\left(K\cdot K\right)=1$. In \cite[Theorem 1.3]{Balazs Barany R^3}
he generalises this result to $K\cdot K\cdot K$ as he shows that
if $\dim_{H}K>\frac{1}{3}$ then $\dim_{H}\left(K\cdot K\cdot K\right)=1$.

\begin{thm}
Let $\left\{ S_{e}:e\in\mathcal{E}\right\} $ be a strongly connected
GD-IFS in $\mathbb{R}$ with attractor $\left(K_{1},\ldots,K_{q}\right)$
and let $j\in\mathcal{V}$. Then $\dim_{H}\left(K_{j}\cdot K_{j}\right)=\min\left\{ 2\dim_{H}(K_{j}),1\right\} $.
In particular, if $\dim_{H}K_{j}\geq\frac{1}{2}$ then $\dim_{H}\left(K_{j}\cdot K_{j}\right)=1$.
If $\dim_{H}K_{j}>\frac{1}{3}$ then $\dim_{H}\left(K_{j}\cdot K_{j}\cdot K_{j}\right)=1$.
\end{thm}

\begin{proof}
By \cite[Theorem 8.10]{Mattilakonyv} and \cite[Theorem 3.2]{Falconer Techniques}
$\dim_{H}\left(K_{j}\times K_{j}\right)=2\dim_{H}(K_{j})$. Since
multiplication is a locally Lipschitz map the upper bound $\dim_{H}\left(K_{j}\cdot K_{j}\right)\leq\min\left\{ 2\dim_{H}(K_{j}),1\right\} $
follows trivially. By Corollary \ref{cor:finite approx cor} for $\varepsilon>0$
we can find an SS-IFS with attractor $K$ such that $K\subseteq K_{j}$,
$\dim_{H}K_{j}-\varepsilon<\dim_{H}K$, and $T_{i}=Id_{\mathbb{R}}$
for every $i$. By choosing $\varepsilon$ small enough we may assume
that if $\dim_{H}K_{j}>\frac{1}{3}$ then $\dim_{H}K>\frac{1}{3}$.
Then we can apply \cite[Corollary 2.9, Theorem 1.3]{Balazs Barany R^3}
to $K$. Since $\varepsilon>0$ is arbitrary the theorem follows.
\end{proof}

Peres and Shmerkin \cite[Theorem 2]{Peres-Schmerkin Resonance between Cantor sets}
showed a similar result about arithmetic sums of self-similar sets.
They proved that if there are two SS-IFS $\left\{ S_{i}\right\} _{i=1}^{m}$
and $\left\{ \widehat{S}_{i}\right\} _{i=1}^{\widehat{m}}$ in $\mathbb{R}$
with attractors $K$ and $\widehat{K}$ such that $\log r_{1}/\log\widehat{r}_{1}\notin\mathbb{Q}$
then $\dim_{H}\left(K+\widehat{K}\right)=\min\left\{ \dim_{H}(K)+\dim_{H}(\widehat{K}),1\right\} $.

\begin{thm}
Let $\left\{ S_{e}:e\in\mathcal{E}\right\} $ and $\left\{ \widehat{S}_{e}:e\in\widehat{\mathcal{E}}\right\} $
be GD-IFS in $\mathbb{R}$ with attractors $\left(K_{1},\ldots,K_{q}\right)$
and $\left(\widehat{K}_{1},\ldots,\widehat{K}_{\widehat{q}}\right)$,
assume there exist $j\in\mathcal{V}$, $\widehat{j}\in\widehat{\mathcal{V}}$
and $\boldsymbol{\mathbf{e}}\in\mathcal{C}_{j}$, $\boldsymbol{\mathbf{f}}\in\widehat{\mathcal{C}}_{\widehat{j}}$
such that $\log r_{\boldsymbol{\mathbf{e}}}/\log r_{\boldsymbol{\mathbf{f}}}\notin\mathbb{Q}$,
then $\dim_{H}\left(K_{j}+\widehat{K}_{\widehat{j}}\right)=\min\left\{ \dim_{H}(K_{j})+\dim_{H}(\widehat{K}_{\widehat{j}}),1\right\} $.
\end{thm}

\begin{proof}
By \cite[Theorem 8.10]{Mattilakonyv} and \cite[Theorem 3.2]{Falconer Techniques}
$\dim_{H}\left(K_{j}\times\widehat{K}_{\widehat{j}}\right)=\dim_{H}(K_{j})+\dim_{H}(\widehat{K}_{\widehat{j}})$.
Since addition is a locally Lipschitz map the upper bound
\[
\dim_{H}\left(K_{j}+\widehat{K}_{\widehat{j}}\right)\leq\min\left\{ \dim_{H}(K_{j})+\dim_{H}(\widehat{K}_{\widehat{j}}),1\right\} 
\]
follows trivially. The opposite inequality follows from \cite[Theorem 2]{Peres-Schmerkin Resonance between Cantor sets}
and Remark \ref{quocient remark}.
\end{proof}

The dimension approximation theorems have consequences in connection
with \foreignlanguage{english}{Furstenberg}'s \foreignlanguage{english}{`dimension
conservation'}. \foreignlanguage{english}{If $f:A\longrightarrow\mathbb{R}^{d_{2}}$
is a Lipschitz map where $A\subseteq\mathbb{R}^{d}$ we say that $f$
is \textsl{dimension conserving} if, for some $\delta\geq0$,
\[
\delta+\dim_{H}\left\{ y\in f(A):\dim_{H}(f^{-1}(y))\geq\delta\right\} \geq\dim_{H}A
\]
with that convention that $\dim_{H}(\emptyset)=-\infty$ so that $\delta$
cannot be chosen too large. Furstenberg also introduces the definition
of a `homogeneous set' \cite[Definition 1.4]{Furstenberg-dimension conservation}.
The main theorem of that paper \cite[Theorem 6.2]{Furstenberg-dimension conservation}
states that the restriction of a linear map to a homogeneous compact
set is dimension conserving. It is pointed out in the paper that if
$K$ is a self-similar set, $\mathcal{T}$ has only one element and
the SSC is satisfied then $K$ is homogeneous. It follows from the
definition of homogeneous sets that $K_{j}$ is homogeneous even if
}$\left(K_{1},\ldots,K_{q}\right)$ is a graph directed attractor
of a strongly connected GD-IFS,\foreignlanguage{english}{ $\mathcal{T}_{j,G}$
is finite and the SSC is satisfied. Thus for such $K$ the restriction
of any linear map to $K$ is dimension conserving.}

Applying the dimension approximation results does not give exact dimension
conservation. However, we can deduce \foreignlanguage{english}{`almost
dimension conservation'}.

\begin{thm}
Let $\left\{ S_{e}:e\in\mathcal{E}\right\} $ be a strongly connected
GD-IFS in $\mathbb{R}^{d}$ with attractor $\left(K_{1},\ldots,K_{q}\right)$,
let $j\in\mathcal{V}$, let $L:\mathbb{R}^{d}\longrightarrow\mathbb{R}^{d_{2}}$
be a linear map ($d_{2}\in\mathbb{N}$) and assume that $\left|\mathcal{T}_{j,G}\right|<\infty$.
Then there exists $\delta\geq0$ such that for every $\varepsilon>0$
\[
\delta+\dim_{H}\left\{ y\in L(K_{j}):\dim_{H}(L^{-1}(y)\cap K_{j})\geq\delta-\varepsilon\right\} \geq\dim_{H}K_{j}.
\]
\end{thm}

\begin{proof}
By Corollary \ref{cor:finite approx cor}, for all $n\in\mathbb{N}$
we can find an SS-IFS that satisfies the SSC with attractor $K^{(n)}$
such that $K^{(n)}\subseteq K_{j}$, $\dim_{H}K_{j}-1/n<\dim_{H}K^{(n)}$,
and the orthogonal part of the similarities are $Id_{\mathbb{R}^{d}}$.
Hence $L\vert_{K^{(n)}}$ is a dimension conserving map by \cite[Theorem 6.2]{Furstenberg-dimension conservation}
and the fact that $K^{(n)}$ is homogeneous. Thus there exists $\delta_{n}\geq0$
such that 
\[
\delta_{n}+\dim_{H}\left\{ y\in L(K^{(n)}):\dim_{H}(L^{-1}(y)\cap K^{(n)})\geq\delta_{n}\right\} \geq\dim_{H}K^{(n)}>\dim_{H}K_{j}-1/n.
\]
We can take a convergent subsequence $\delta_{n_{k}}$ of $\delta_{n}$
with limit $\delta$. Let $\varepsilon>0$ be arbitrary. Then 
\[
\delta_{n_{k}}+\dim_{H}\left\{ y\in L(K_{j}):\dim_{H}(L^{-1}(y)\cap K_{j})\geq\delta-\varepsilon\right\} >\dim_{H}K_{j}-1/n_{k}
\]
whenever $\delta_{n_{k}}\geq\delta-\varepsilon$. Taking the limit
on both sides we get the conclusion of the Theorem.
\end{proof}

When $G_{d,l}$ has a dense orbit under the action of $\mathcal{T}_{j,G}$
where $l\geq\dim_{H}(K_{j})$ is the rank of the linear map, then
we can prove dimension conservation.
\begin{thm}
Let $\left\{ S_{e}:e\in\mathcal{E}\right\} $ be a strongly connected
GD-IFS in $\mathbb{R}^{d}$ with attractor $\left(K_{1},\ldots,K_{q}\right)$,
let $j\in\mathcal{V}$ and \foreignlanguage{english}{$U$ be an open
neighbourhood of $K_{j}$. If} \foreignlanguage{english}{$g:U\longrightarrow\mathbb{R}^{d_{2}}$
is a continuously differentiable map ($d_{2}\in\mathbb{N}$) such
that $\mathrm{rank}(g'(x))=l$ for every $x\in K_{j}$ where $\dim_{H}(K_{j})\leq l$
and there exists $M\in G_{d,l}$ such that the set $\left\{ O(M):O\in\mathcal{T}_{j,G}\right\} $
is dense in $G_{d,l}$ for some $1\leq l<d$ then $g\vert_{K_{j}}$
is a dimension conserving map.}
\end{thm}

This follows from Corollary \ref{cor:hoch-shmer cor} taking $\delta=0$\foreignlanguage{english}{.}

\selectlanguage{english}%
On the plane either $\left|\mathcal{T}_{j,G}\right|<\infty$ or $\left|\mathcal{T}_{j,G}\right|=\infty$
implies that $\left\{ O(M):O\in\mathcal{T}_{j,G}\right\} $ is dense
in $G_{2,1}$ for every $M\in G_{2,1}$. Falconer and Jin \cite[Theorem 4.8]{Falconer-Jin}
showed a property in some sense stronger than `almost dimension conservation'
for the projections of a self-similar set with infinite transformation
group $\mathcal{T}$ when $1<\dim_{H}(K_{j})$. We generalise their
result to graph directed attractors with no separation condition.\foreignlanguage{british}{
Let $\Pi_{\theta}$ denote the orthogonal projection map onto the
line $\left\{ (\lambda\cos\theta,\lambda\sin\theta):\lambda\in\mathbb{R}\right\} $.}

\selectlanguage{british}%

\begin{thm}
Let $\left\{ S_{e}:e\in\mathcal{E}\right\} $ be a strongly connected
GD-IFS in $\mathbb{R}^{2}$ with attractor $\left(K_{1},\ldots,K_{q}\right)$
and assume that $\left|\mathcal{T}_{j,G}\right|=\infty$. If $\dim_{H}(K_{j})>1$
then there exists $E\subseteq[0,\pi)$ with $\dim_{H}E=0$ such that
for all $\theta\in[0,\pi)\setminus E$ and for all $\varepsilon>0$
\[
\mathcal{L}^{1}\left\{ y\in\Pi_{\theta}(K_{j}):\dim_{H}(\Pi_{\theta}^{-1}(y)\cap K_{j})\geq\dim_{H}(K_{j})-1-\varepsilon\right\} >0.
\]
\end{thm}

\begin{proof}
By Theorem \ref{thm:planar dense thm} there exist SS-IFSs that satisfy
the SSC with attractor $K^{(n)}$ such that $K^{(n)}\subseteq K_{j}$,
$\dim_{H}K_{j}-1/n<\dim_{H}K^{(n)}$, the transformation group $\mathcal{T}^{(n)}$
is dense in $\mathcal{T}_{j,G}\cap\mathbb{SO}_{2}$ and $T_{i}=O^{(n)}$,
$r_{i}=r^{(n)}$ for every $i\in\mathcal{I}^{(n)}$. Then by \cite[Theorem 4.6]{Falconer-Jin}
whenever $\dim_{H}K_{j}-1/n>1$ there exists $E^{(n)}\subseteq[0,\pi)$
with $\dim_{H}E^{(n)}=0$ such that for all $\theta\in[0,\pi)\setminus E^{(n)}$
and for all $\varepsilon>0$ 
\[
\mathcal{L}^{1}\left\{ y\in\Pi_{\theta}(K^{(n)}):\dim_{H}(\Pi_{\theta}^{-1}(y)\cap K^{(n)})\geq\dim_{H}(K^{(n)})-1-\varepsilon/2\right\} >0.
\]
Let $E=\bigcup_{n=1}^{\infty}E^{(n)}$. Then taking $1/n\leq\varepsilon/2$
and $\theta\in[0,\pi)\setminus E$ it follows that
\[
\mathcal{L}^{1}\left\{ y\in\Pi_{\theta}(K_{j}):\dim_{H}(\Pi_{\theta}^{-1}(y)\cap K_{j})\geq\dim_{H}(K_{j})-1-\varepsilon\right\} >0.
\]
\end{proof}

\subsection{Approximation theorems for subshifts of finite type\label{subsec:SSOFT section}}

Subshifts of finite type and graph directed attractors are the same
in some sense (see \cite[Proposition 2.5, Proposition 2.6]{Farkas-Fraser SSOFT}).
Thus we can extend our results to subshifts of finite type. Let $\mathcal{J}=\left\{ 0,1,\dots,M-1\right\} $
be a finite alphabet and $A$ be an $M\times M$ transition matrix
indexed by $\mathcal{J}\times\mathcal{J}$ with entries in $\left\{ 0,1\right\} $.
We define the subshift of finite type corresponding to $A$ as
\[
\Sigma_{A}=\left\{ \alpha=(\alpha_{0},\alpha_{1},\dots)\in\mathcal{J}^{\mathbb{N}}:A_{\alpha_{i},\alpha_{i+1}}=1\,\mathrm{for\,all}\,i=0,1,\dots\right\} .
\]
We say $\Sigma_{A}$ is irreducible (or transitive) if the matrix
$A$ is irreducible, which means that for all pairs $i,j\in\mathcal{J}$,
there exists a positive integer $n$ such that $\left(A^{n}\right)_{i,j}>0$.
To each $i\in\mathcal{J}$ associate a contracting similarity map
$S_{i}(x)=r_{i}\cdot T_{i}(x)+v_{i}$ on $\mathbb{R}^{d}$ where $r_{i}\in(0,1)$,
$T_{i}\in\mathbb{O}_{d}$ and $v_{i}\in\mathbb{R}^{d}$. For $\alpha=(\alpha_{0},\alpha_{1},\dots)\in\mathcal{J}^{\mathbb{N}}$
and $k\in\mathbb{N}$ we write $\alpha\vert_{k}=(\alpha_{0},\dots,\alpha_{k})\in\mathcal{J}^{k}$
and for $\boldsymbol{\mathbf{i}}=(i_{0},\dots,i_{k-1})\in\mathcal{J}^{k}$
we write
\[
S_{\boldsymbol{\mathbf{i}}}=S_{i_{0}}\circ\dots S_{i_{k-1}}.
\]
Then $\Pi(\alpha)=\lim_{k\rightarrow\infty}S_{\alpha\vert_{k}}(0)$
exists for every $\alpha=(\alpha_{0},\alpha_{1},\dots)\in\mathcal{J}^{\mathbb{N}}$.
For a given subshift of finite type we study the set $F_{A}=\Pi(\Sigma_{A})\subseteq\mathbb{R}^{d}$.
For $j\in\mathcal{J}$ let $F_{A}^{j}=\Pi(\Sigma_{A}^{j})$ where
$\Sigma_{A}^{j}=\left\{ (\alpha_{0},\alpha_{1},\dots)\in\Sigma_{A}:\alpha_{0}=j\right\} $.
Note that if $\Sigma_{A}$ is an irreducible then $\dim_{H}F_{A}=\dim_{H}F_{A}^{j}$
for every $j\in\mathcal{J}$.

If $\Sigma_{A}$ is an irreducible subshift of finite type then there
exists a strongly connected GD-IFS with attractor $(F_{A}^{1},\dots,F_{A}^{M-1})$,
see \cite[Proposition 2.6]{Farkas-Fraser SSOFT}. For the completeness
we include the construction. Let $\mathcal{J}$ be the set of vertices.
We draw a directed edge $e=e_{i,j}$ from $i$ to $j$ if $A_{i,j}=1$,
let $S_{e}=S_{i}$ and let $\mathcal{E}=\left\{ e_{i,j}:i,j\in\mathcal{J},A_{i,j}=1\right\} $.
If $A$ is irreducible then the graph $G(\mathcal{J},\mathcal{E})$
is strongly connected. We have that
\[
F_{A}^{i}=\bigcup_{j\in\mathcal{J},A_{i,j}=1}S_{i}(F_{A}^{j})=\bigcup_{j\in\mathcal{J}}\bigcup_{e\in\mathcal{E}_{i,j}}S_{e}(F_{A}^{j}).
\]
Then the set of directed cycles in $G(\mathcal{J},\mathcal{E})$ is
\[
\mathcal{C}_{j}=\bigcup_{k=1}^{\infty}\mathcal{E}_{j,j}^{k}=\bigcup_{k=1}^{\infty}\biggr\{(\alpha_{0},\dots,\alpha_{k-1})\in\mathcal{J}^{k}:A_{j,\alpha_{0}}=1,A_{\alpha_{k-1},j}=1
\]
\[
\mathrm{and}\,A_{\alpha_{i},\alpha_{i+1}}=1\,\mathrm{for\,all}\,i=0,1,\dots k-2\biggl\}.
\]
Hence we define the \textsl{$j$-th transformation group $\mathcal{T}_{j,A}^{G}$
of $\Sigma_{A}$} to be the group generated by the semigroup
\[
\left\{ T_{\alpha}:k\in\mathbb{N},\alpha=(\alpha_{0},\dots,\alpha_{k-1})\in\mathcal{J}^{k},A_{j,\alpha_{0}}=1\,\mathrm{and}\,A_{\alpha_{i},\alpha_{i+1}}=1\,\mathrm{for\,all}\,i=0,1,\dots k-2\right\} .
\]
Note that if $A_{j,\alpha_{0}}=1$ then $T_{\alpha_{0}}=T_{j}$. Now
we are ready to reformulate Theorem \ref{thm: MAIN} for subshifts
of finite type.

\begin{thm}
Let $\Sigma_{A}$ be an irreducible subshift of finite type and $j\in\mathcal{J}$.
Then for every $\varepsilon>0$ there exists an SS-IFS $\left\{ S_{i}\right\} _{i=1}^{m}$
that satisfies the SSC, with attractor $K$ such that $K\subseteq F_{A}^{j}\subseteq F_{A}$,
$\dim_{H}F_{A}-\varepsilon<\dim_{H}K$ and the transformation group
$\mathcal{T}$ of $\left\{ S_{i}\right\} _{i=1}^{m}$ is dense in
$\mathcal{T}_{j,A}^{G}$.
\end{thm}

We can also state the subshift of finite type analogue of Theorem
\ref{thm:approx trafo thm}.

\begin{thm}
Let $\Sigma_{A}$ be an irreducible subshift of finite type and $j\in\mathcal{J}$.
Then for every $\varepsilon>0$ and $O\in\overline{\mathcal{T}_{j,A}^{G}}$
there exist $r\in(0,1)$ and an SS-IFS $\left\{ S_{i}\right\} _{i=1}^{m}$
that satisfies the SSC, with attractor $K$ such that $K\subseteq F_{A}^{j}\subseteq F_{A}$,
$\dim_{H}F_{A}-\varepsilon<\dim_{H}K$, and $\left\Vert T_{i}-O\right\Vert <\varepsilon$,
$r_{i}=r$ for every $i\in\left\{ 1,\dots,m\right\} $.
\end{thm}

As in the case of Corollary \ref{cor:finite approx cor} for graph
directed attractors we can conclude the following for subshifts of
finite type.

\begin{cor}
Let $\Sigma_{A}$ be an irreducible subshift of finite type, let $j\in\mathcal{J}$
and assume that $\mathcal{T}_{j,A}^{G}$ is a finite group. Then for
every $\varepsilon>0$ and $O\in\mathcal{T}_{j,A}^{G}$ there exist
$r\in(0,1)$ and an SS-IFS $\left\{ S_{i}\right\} _{i=1}^{m}$ that
satisfies the SSC, with attractor $K$ such that $K\subseteq F_{A}^{j}\subseteq F_{A}$,
$\dim_{H}F_{A}-\varepsilon<\dim_{H}K$, and $T_{i}=O$, $r_{i}=r$
for every $i\in\left\{ 1,\dots,m\right\} $.
\end{cor}

In the plane we can formulate the subshift of finite type analogue
of Theorem \ref{thm:planar dense thm}.

\begin{thm}
Let $\Sigma_{A}$ be an irreducible subshift of finite type in $\mathbb{R}^{2}$
and let $j\in\mathcal{J}$. Then for every $\varepsilon>0$ there
exist $r\in(0,1)$, an orthogonal transformation $O\in\mathcal{T}_{j,A}^{G}\cap\mathbb{SO}_{2}$
and an SS-IFS $\left\{ S_{i}\right\} _{i=1}^{m}$ that satisfies the
SSC, with attractor $K$ such that $K\subseteq F_{A}^{j}\subseteq F_{A}$,
$\dim_{H}F_{A}-\varepsilon<\dim_{H}K$, the transformation group $\mathcal{T}$
of $\left\{ S_{i}\right\} _{i=1}^{m}$ is dense in $\mathcal{T}_{j,A}^{G}\cap\mathbb{SO}_{2}$
and $T_{i}=O$, $r_{i}=r$ for every $i\in\left\{ 1,\dots,m\right\} $.
\end{thm}

As a consequence we can restate every result of Section \ref{subsec:Application GDA}
for subshifts of finite type and the proofs proceed similarly. We
omit the restatement of those results as they are very similar to
those in Section \ref{subsec:Application GDA}. However, we note that
Fraser and Pollicott \cite[Theorem 2.10]{Fraser-Pollicott SSOFT}
proved the subshift of finite type version of Theorem \ref{thm:hoch-shmer}
for systems satisfying the `strong separation property'. In \cite[Theorem 2.7]{Fraser-Pollicott SSOFT}
they also proved the subshift of finite type version of Theorem \ref{thm:distance set R2}
for the even more general case of conformal systems rather then similarities.

\section{Proof of the approximation theorems\label{sec:Proofs}}

We prove our main approximation theorems in this section
\selectlanguage{english}%
\begin{prop}
\label{lem:Jordanos lem}If $T\in\mathbb{O}_{d}$ then for all $N\in\mathbb{N}$
there exists $k\in\mathbb{N}$, $k\geq N$, such that the group generated
by $T^{k}$ is dense in the group generated by $T$.
\end{prop}

\selectlanguage{british}%
See \cite[Poposition 2.2]{Farkas-Linim-only}.
\selectlanguage{english}%
\begin{lem}
\label{lem:FP change}Let $S_{1},\ldots,S_{k}:\mathbb{R}^{d}\longrightarrow\mathbb{R}^{d}$
$(k\geq2)$ be contracting similarities such that $S_{1}$ and $S_{2}$
have no common fixed point. Then there exist $F_{1},\ldots,F_{k}:\mathbb{R}^{d}\longrightarrow\mathbb{R}^{d}$
such that $F_{1}=S_{1}$, $F_{2}=S_{2}$, for each $i\in\left\{ 3,\ldots,k\right\} $
either $F_{i}=S_{1}^{k_{i}}\circ S_{i}$ or $F_{i}=S_{2}^{k_{i}}\circ S_{i}$
for some $k_{i}\in\mathbb{N}$, and $F_{i}$ and $F_{j}$ have no
common fixed point for all $i,j\in\left\{ 1,\ldots,k\right\} $, $i\neq j$.
\end{lem}

\selectlanguage{british}%
See \cite[Lemma 7.2]{Farkas-Linim-only}.

\begin{lem}
\label{lem:dense piece lem}Let $\left\{ S_{e}:e\in\mathcal{E}\right\} $
be a strongly connected GD-IFS in $\mathbb{R}^{d}$ with attractor
$\left(K_{1},\ldots,K_{q}\right)$ and let $j\in\mathcal{V}$. Then
\[
\bigcup\left\{ K_{\boldsymbol{\mathbf{e}}}:\boldsymbol{\mathbf{e}}\in\mathcal{C}_{j},\mathrm{diam}(K_{\boldsymbol{\mathbf{e}}})<\gamma\right\} 
\]
is dense in $K_{j}$ for every $\gamma>0$.
\end{lem}

\begin{proof}
Fix $\gamma>0$. Let $x\in K_{j}$ and $\gamma>\delta>0$ arbitrary
and we find a point $y\in\bigcup\left\{ K_{\boldsymbol{\mathbf{e}}}:\boldsymbol{\mathbf{e}}\in\mathcal{C}_{j},\mathrm{diam}(K_{\boldsymbol{\mathbf{e}}})<\gamma\right\} $
which is $\delta$-close to $x$. For every $n$ we have that $\bigcup_{i=1}^{q}\bigcup_{\boldsymbol{\mathbf{f}}\in\mathcal{E}_{j,i}^{n}}K_{\boldsymbol{\mathbf{f}}}$
is a cover of $K_{j}$. For $n$ large enough $\mathrm{diam}(K_{\boldsymbol{\mathbf{f}}})<\delta<\gamma$
for every $\boldsymbol{\mathbf{f}}\in\bigcup_{i=1}^{q}\mathcal{E}_{j,i}^{n}$.
Let $i\in\mathcal{V}$ and $\boldsymbol{\mathbf{f}}_{1}\in\mathcal{E}_{j,i}^{n}$
be such that $x\in K_{\boldsymbol{\mathbf{f}}_{1}}$. It follows that
$\left\Vert x-y\right\Vert <\delta$ for every $y\in K_{\boldsymbol{\mathbf{f}}_{1}}$.
Since $\left\{ S_{e}:e\in\mathcal{E}\right\} $ is strongly connected
$\bigcup_{k=1}^{\infty}\mathcal{E}_{i,j}^{k}\neq\emptyset$, so let
$\boldsymbol{\mathbf{f}}_{2}\in\bigcup_{1=N}^{\infty}\mathcal{E}_{i,j}^{k}$.
Then $\boldsymbol{\mathbf{e}}=\boldsymbol{\mathbf{f}}_{1}*\boldsymbol{\mathbf{f}}_{2}\in\mathcal{C}_{j}$
and $K_{\boldsymbol{\mathbf{e}}}\subseteq K_{\boldsymbol{\mathbf{f}}}$,
thus $\mathrm{diam}(K_{\boldsymbol{\mathbf{e}}})<\gamma$ while $\left\Vert x-y\right\Vert <\delta$
for every $y\in K_{\boldsymbol{\mathbf{e}}}$.
\end{proof}

\begin{lem}
\label{lem: non singleton lem}Let $\left\{ S_{e}:e\in\mathcal{E}\right\} $
be a strongly connected GD-IFS in $\mathbb{R}^{d}$ with attractor
$\left(K_{1},\ldots,K_{q}\right)$. If \foreignlanguage{english}{$K_{j}$
contains at least two points for some $j\in\mathcal{V}$ then there
exist $\boldsymbol{\mathbf{e}},\boldsymbol{\mathbf{f}}\in\mathcal{C}_{j}$
such that $S_{\boldsymbol{\mathbf{e}}}$ and $S_{\boldsymbol{\mathbf{f}}}$
have no common fixed point.}
\end{lem}

\begin{proof}
Assume that $S_{\boldsymbol{\mathbf{e}}}$ have the same fixed point
$x$ for every $\boldsymbol{\mathbf{e}}\in\mathcal{C}_{j}$. Then
\[
\bigcup\left\{ K_{\boldsymbol{\mathbf{e}}}:\boldsymbol{\mathbf{e}}\in\mathcal{C}_{j},\mathrm{diam}(K_{\boldsymbol{\mathbf{e}}})<\gamma\right\} \subseteq B(x,\gamma)
\]
for all $\gamma>0$ and it follows from Lemma \ref{lem:dense piece lem}
that $K_{j}=\left\{ x\right\} $ which is a contradiction. 
\end{proof}

\begin{lem}
\label{lem:finitegen lem}Let $\left\{ S_{e}:e\in\mathcal{E}\right\} $
be a strongly connected GD-IFS in $\mathbb{R}^{d}$ and $j\in\mathcal{V}$.
Then there exists a finite set of cycles $\left\{ \boldsymbol{\mathbf{e}}_{1},\dots,\boldsymbol{\mathbf{e}}_{k}\right\} \subseteq\mathcal{C}_{j}$
such that $T_{\boldsymbol{\mathbf{e}}_{1}},\dots,T_{\boldsymbol{\mathbf{e}}_{k}}$
generate $\mathcal{T}_{i,G}$.
\end{lem}

\begin{proof}
For every $i\in\mathcal{V}$, $i\neq j$ let us fix a directed path
$a_{i}$ from $j$ to $i$ and a directed path $b_{i}$ from $i$
to $j$. For $i=j$ let $a_{j}$ and $b_{j}$ be the empty path. We
claim that the finite set of cycles
\[
\bigcup_{i,l\in\mathcal{V}}\left\{ a_{i}*e*b_{l}:e\in\mathcal{E}_{i,l}\right\} \bigcup\left\{ a_{i}*b_{i}:i\in\mathcal{V}\right\} \subseteq\mathcal{C}_{j}
\]
satisfies the lemma. Let $\boldsymbol{\mathbf{e}}=(e_{1},\ldots,e_{n})\in\mathcal{C}_{j}$
be an arbitrary cycle that visits the vertices $i_{0},i_{1},\dots,i_{n}$
respectively. Then with that convention that $T_{\emptyset}=Id_{\mathbb{R}^{d}}$
we have that 
\[
\begin{aligned}T_{\boldsymbol{\mathbf{e}}} & =T_{e_{1}}\circ\dots\circ T_{e_{n}}\\
 & =T_{b_{i_{0}}}^{-1}\circ T_{e_{1}}\circ T_{b_{i_{1}}}\circ\dots\circ T_{b_{i_{n-1}}}^{-1}\circ T_{e_{n}}\circ T_{b_{i_{n}}}\\
 & =T_{a_{i_{0}}*b_{i_{0}}}^{-1}\circ T_{a_{i_{0}}*e_{1}*b_{i_{1}}}\circ\dots\circ T_{a_{i_{n-1}}*b_{i_{n-1}}}^{-1}\circ T_{a_{i_{n-1}}*e_{n}*b_{i_{n}}}
\end{aligned}
\]
which completes the proof.
\end{proof}

The proof of the following lemma is based on the idea of the beginning
of the proof of Peres and Shmerkin \cite[Theorem 2]{Peres-Schmerkin Resonance between Cantor sets}.

\begin{lem}
\label{lem:justdimapprox lem}Let $\left\{ S_{e}:e\in\mathcal{E}\right\} $
be a strongly connected GD-IFS in $\mathbb{R}^{d}$ with attractor
$\left(K_{1},\ldots,K_{q}\right)$, let $j\in\mathcal{V}$ and let
$\boldsymbol{\mathbf{e}}\in\mathcal{C}_{j}$. Then for every $\varepsilon>0$
there exists an SS-IFS $\left\{ \widehat{S_{i}}\right\} _{i=1}^{n}$
that satisfies the SSC, with attractor $\widehat{K}$ such that $\widehat{K}\subseteq K_{\boldsymbol{\mathbf{e}}}$
and $\dim_{H}K_{j}-\varepsilon<\dim_{H}\widehat{K}$.
\end{lem}

\selectlanguage{english}%
\begin{proof}
Let $t=\dim_{H}K_{j}=\dim_{H}K_{\boldsymbol{\mathbf{e}}}$. Since
$K_{j}$ has at least two points it follows that $K_{j}$ has infinitely
many points and so $\mathcal{H}^{0}(K_{j})=\infty$. On the other
hand, since \foreignlanguage{british}{$\left\{ S_{e}:e\in\mathcal{E}\right\} $
is strongly connected }$\mathcal{H}^{t}(K_{j})<\infty$ by \cite[Thm 3.2]{Falconer Techniques}.
Thus $t>0$ and hence without the loss of generality we can assume
that $t>\varepsilon>0$. Since $\mathcal{H}^{t-\frac{\varepsilon}{2}}(K_{\boldsymbol{\mathbf{e}}})=\infty$
we can find $\delta>0$ such that for any $3\delta$-cover $\mathcal{U}$
of $K_{\boldsymbol{\mathbf{e}}}$ we have that $\sum_{U\in\mathcal{U}}\mathrm{diam}(U)^{t-\frac{\varepsilon}{2}}>1$.
Let $r_{\mathrm{min}}=\min\left\{ r_{e}:e\in\mathcal{E}\right\} <1$
and let
\[
\mathcal{J}=\left\{ \boldsymbol{\mathbf{f}}\in\bigcup_{i\in\mathcal{V}}\bigcup_{k=1}^{\infty}\mathcal{E}_{j,i}^{k}:\exists\boldsymbol{\mathbf{g}}\in\bigcup_{i\in\mathcal{V}}\bigcup_{k=1}^{\infty}\mathcal{E}_{j,i}^{k},\boldsymbol{\mathbf{f}}=\boldsymbol{\mathbf{e}}*\boldsymbol{\mathbf{g}},\,r_{\mathrm{min}}\delta\leq\mathrm{diam}(K_{\boldsymbol{\mathbf{f}}})<\delta\right\} .
\]
Then $\left\{ K_{\boldsymbol{\mathbf{f}}}:\boldsymbol{\mathbf{f}}\in\mathcal{J}\right\} $
is a cover of $K_{\boldsymbol{\mathbf{e}}}$. Let $\boldsymbol{\mathbf{f}}_{1},\ldots,\boldsymbol{\mathbf{f}}_{n}\in\mathcal{J}$
be such that $K_{\boldsymbol{\mathbf{f}}_{1}},\ldots,K_{\boldsymbol{\mathbf{f}}_{n}}$
is a maximal pairwise disjoint sub-collection of $\left\{ K_{\boldsymbol{\mathbf{f}}}:\boldsymbol{\mathbf{f}}\in\mathcal{J}\right\} $.
Let $U_{i}$ be the $\delta$-neighbourhood of $K_{\boldsymbol{\mathbf{f}}_{i}}$
for $i\in\left\{ 1,\ldots,n\right\} $. By the maximality $\left\{ U_{i}:i\in\left\{ 1,\ldots,n\right\} \right\} $
is a $3\delta$-cover of $K_{\boldsymbol{\mathbf{e}}}$. Hence by
the choice of $\delta$
\[
\sum_{i=1}^{n}\left(3\delta\right)^{t-\frac{\varepsilon}{2}}\geq\sum_{i=1}^{n}\left(\mathrm{diam}(U_{i})\right)^{t-\frac{\varepsilon}{2}}>1.
\]
It follows that
\begin{equation}
n\geq\left(3\delta\right)^{-\left(t-\frac{\varepsilon}{2}\right)}.\label{eq:n>...}
\end{equation}
\foreignlanguage{british}{Assume that the paths $\boldsymbol{\mathbf{f}}_{1},\ldots,\boldsymbol{\mathbf{f}}_{n}$
end in vertices $i_{1},\dots,i_{n}$ respectively. For every $i\in\mathcal{V}$
let us fix a directed path $b_{i}$ from $i$ to $j$. Let $c_{\mathrm{min}}=\min_{i\in\mathcal{V}}\left\{ r_{b_{i}}\frac{\mathrm{diam}(K_{j})}{\mathrm{diam}(K_{i})}\right\} $.
Then
\begin{equation}
\mathrm{diam}(K_{\boldsymbol{\mathbf{f}}_{l}*b_{i_{l}}})\geq c_{\mathrm{min}}\cdot\mathrm{diam}(K_{\boldsymbol{\mathbf{f}}_{l}})\geq c_{\mathrm{min}}\cdot r_{\mathrm{min}}\cdot\delta\label{eq:diam lowerbound}
\end{equation}
and $\boldsymbol{\mathbf{f}}_{l}*b_{i_{l}}\in\mathcal{C}_{j}$ for
every $l\in\left\{ 1,\dots,n\right\} $.}

Let $\widehat{K}$ be the attractor of the SS-IFS $\left\{ S_{\boldsymbol{\mathbf{f}}_{l}*b_{i_{l}}}\right\} _{l=1}^{n}$.
Then $\widehat{K}\subseteq K_{\boldsymbol{\mathbf{e}}}$, the SS-IFS
$\left\{ S_{\boldsymbol{\mathbf{f}}_{l}*b_{i_{l}}}\right\} _{l=1}^{n}$
satisfies the SSC and
\[
\dim_{H}\widehat{K}\geq\frac{\log(\frac{1}{n})}{\log(\frac{c_{\mathrm{min}}\cdot r_{\mathrm{min}}\cdot\delta}{\mathrm{diam}(K_{j})})}\geq\frac{-\left(t-\frac{\varepsilon}{2}\right)\cdot\log(3)-\left(t-\frac{\varepsilon}{2}\right)\cdot\log(\delta)}{\log(\mathrm{diam}(K_{j}))-\log(c_{\mathrm{min}})-\log(r_{\mathrm{min}})-\log(\delta)}
\]
by (\ref{eq:n>...}), (\ref{eq:diam lowerbound}) and because the
similarity dimension of $\left\{ S_{\boldsymbol{\mathbf{f}}_{l}*b_{i_{l}}}\right\} _{l=1}^{n}$
is $\dim_{H}\widehat{K}$, see (\ref{eq: S-dim sum}). So, by choosing
$\delta$ small enough, $\dim_{H}\widehat{K}>t-\varepsilon$. Hence
the SS-IFS $\left\{ \widehat{S_{i}}\right\} _{i=1}^{n}=\left\{ S_{\boldsymbol{\mathbf{f}}_{l}*b_{i_{l}}}\right\} _{l=1}^{n}$
satisfies the lemma.
\end{proof}
\selectlanguage{british}%

Similar ideas to the proof of Theorem \ref{thm: MAIN} were used by
Farkas \cite[Proposition 1.8]{Farkas-Linim-only} in the case of self-similar
sets.\\

\selectlanguage{english}%
\noindent \textit{Proof of Theorem }\foreignlanguage{british}{\ref{thm: MAIN}.
According to Lemma \ref{lem: non singleton lem} and Lemma \ref{lem:finitegen lem}
there exist $\boldsymbol{\mathbf{e}}_{1},\dots,\boldsymbol{\mathbf{e}}_{k}\in\mathcal{C}_{j}$
such that }$S_{\boldsymbol{\mathbf{e}}_{1}}$ and $S_{\boldsymbol{\mathbf{e}}_{2}}$
have no common fixed point and \foreignlanguage{british}{$T_{\boldsymbol{\mathbf{e}}_{1}},\dots,T_{\boldsymbol{\mathbf{e}}_{k}}$
generate $\mathcal{T}_{j,G}$ (note that $\boldsymbol{\mathbf{e}}_{1}\in\mathcal{C}_{j}$
can be chosen arbitrarily, see Remark \ref{log r}). Hence by Lemma
\ref{lem:FP change} there exist $\boldsymbol{\mathbf{f}}_{1},\dots,\boldsymbol{\mathbf{f}}_{k}\in\mathcal{C}_{j}$
such that $S_{\boldsymbol{\mathbf{f}}_{i}}$ and $S_{\boldsymbol{\mathbf{f}}_{l}}$
have no common fixed point for every $i,l\in\left\{ 1,\dots,k\right\} $
where $i\neq l$, $\boldsymbol{\mathbf{f}}_{1}=\boldsymbol{\mathbf{e}}_{1}$,
$\boldsymbol{\mathbf{f}}_{2}=\boldsymbol{\mathbf{e}}_{2}$ and $T_{\boldsymbol{\mathbf{f}}_{1}},\dots,T_{\boldsymbol{\mathbf{f}}_{k}}$
generate $\mathcal{T}_{j,G}$. Let $x_{i}$ be the unique fixed point
of $S_{\boldsymbol{\mathbf{f}}_{i}}$ for every $i\in\left\{ 1,\dots,k\right\} $.
Let $d_{\mathrm{min}}=\min\left\{ \left\Vert x_{i}-x_{l}\right\Vert :i,l\in\left\{ 1,\dots,k\right\} ,i\neq l\right\} >0$,
}$r_{\mathrm{max}}=\max\left\{ r_{\boldsymbol{\mathbf{f}}_{i}}:i\in\left\{ 1,\dots,k\right\} \right\} <1$
and $N\in\mathbb{N}$ such that $r_{\mathrm{max}}^{N}\cdot\mathrm{diam}(K_{j})<d_{\mathrm{min}}/2$.
Then $S_{\boldsymbol{\mathbf{f}}_{i}}^{k_{i}}(K_{j})\cap S_{\boldsymbol{\mathbf{f}}_{l}}^{k_{l}}(K_{j})=\emptyset$
for all $i,l\in\left\{ 1,\dots,k\right\} $, $i\neq l$, $k_{i},k_{l}\in\mathbb{N}$,
$k_{i},k_{l}\geq N$. By Proposition \ref{lem:Jordanos lem} for all
$i\in\left\{ 1,\dots,k\right\} $ we can find $k_{i}\in\mathbb{N}$,
$k_{i}\geq N$ such that the group generated by $T_{\boldsymbol{\mathbf{f}}_{i}}^{k_{i}}$
is dense in the group generated by $T_{\boldsymbol{\mathbf{f}}_{i}}$.
It follows that the group generated by $T_{\boldsymbol{\mathbf{f}}_{1}}^{k_{1}},\ldots,T_{\boldsymbol{\mathbf{f}}_{k}}^{k_{k}}$
is dense in \foreignlanguage{british}{$\mathcal{T}_{j,G}$} and $S_{\boldsymbol{\mathbf{f}}_{i}}^{k_{i}}(K)\cap S_{\boldsymbol{\mathbf{f}}_{l}}^{k_{l}}(K)=\emptyset$
for all $i,l\in\mathcal{I}$, $i\neq l$. Let $S_{i}=S_{\boldsymbol{\mathbf{f}}_{i}}^{k_{i}}$
for all $i\in\left\{ 1,\dots,k\right\} $.

Let $F=\bigcup_{i=1}^{k}S_{\boldsymbol{\mathbf{f}}_{i}}^{k_{i}}(K_{j})$.
If $K_{j}=F$ then $\left\{ S_{i}\right\} _{i=1}^{k}$ satisfies the
SSC with attractor $K=K_{j}$ and the proof is complete. So we can
assume that $F\subsetneq K_{j}$. By Lemma \ref{lem:dense piece lem}
we can find $\boldsymbol{\mathbf{e}}\in\mathcal{C}_{j}$ such that
$K_{\boldsymbol{\mathbf{e}}}\cap F=\emptyset$. It follows from Lemma
\ref{lem:justdimapprox lem} \foreignlanguage{british}{that there
exists an SS-IFS $\left\{ \widehat{S_{i}}\right\} _{i=1}^{n}$ that
satisfies the SSC with attractor $\widehat{K}$ such that $\widehat{K}\subseteq K_{\boldsymbol{\mathbf{e}}}$
and $\dim_{H}K_{j}-\varepsilon<\dim_{H}\widehat{K}$.} Let $m=k+n$,
$S_{k+l}=\widehat{S_{i}}$ for all $l\in\left\{ 1,\ldots,n\right\} $
and $K$ be the attractor of the SS-IFS \foreignlanguage{british}{$\left\{ S_{i}\right\} _{i=1}^{m}$}.
Then the transformation group $\mathcal{T}$ of \foreignlanguage{british}{$\left\{ S_{i}\right\} _{i=1}^{m}$}
is dense in \foreignlanguage{british}{$\mathcal{T}_{j,G}$}, $\widehat{K}\subseteq K\subseteq K_{j}$,
$\dim_{H}K_{j}-\varepsilon<\dim_{H}\widehat{K}\leq\dim_{H}K$ and
\foreignlanguage{british}{$\left\{ S_{i}\right\} _{i=1}^{m}$} satisfies
the SSC.$\hfill\square$\\

\selectlanguage{british}%

\begin{rem}
\label{log r}As we noted in the proof of Theorem $\ref{thm: MAIN}$
$\boldsymbol{\mathbf{e}}_{1}\in\mathcal{C}_{j}$ can be chosen arbitrarily
because if there are two maps with different fixed points then at
least one of them have a different fixed point from \foreignlanguage{english}{$S_{\boldsymbol{\mathbf{e}}_{1}}$},
and if you add any further map to a generator set of $\mathcal{T}_{j,G}$
it will remain a generator set of $\mathcal{T}_{j,G}$. The similarity
ratio of $\widehat{S_{1}}$ is $r_{1}=r_{\boldsymbol{\mathbf{f}}_{1}}^{k_{1}}=r_{\boldsymbol{\mathbf{e}}_{1}}^{k_{1}}$.
Thus $\log r_{1}=k_{1}\log r_{\boldsymbol{\mathbf{e}}_{1}}$.
\end{rem}

For $k_{1},\dots,k_{m}\in\mathbb{N}$ such that $\sum_{l=1}^{m}k_{l}=k$
let
\begin{equation}
N(k_{1},\dots,k_{m})=\left|\left\{ (i_{1},\dots,i_{k})\in\mathcal{I}^{k}:\left|\left\{ j:1\leq j\leq k,i_{j}=l\right\} \right|=k_{l}\,\mathrm{for\,every}\,1\leq l\leq m\right\} \right|,\label{eq:N(k_1,...,k_m) def}
\end{equation}
i.e. the number of words in $\mathcal{I}^{k}$ such that the symbol
$l$ appears in the word exactly $k_{l}$ times for every $1\leq l\leq m$.

\begin{lem}
\label{lem:Cheb lem}Let $(p_{1},\dots,p_{m})$ be a probability vector.
Then there exists $c>0$ such that for each $k\in\mathbb{N}$ there
exist $k_{1},\dots,k_{m}\in\mathbb{N}$ such that $\sum_{l=1}^{m}k_{l}=k$
and
\[
N(k_{1},\dots,k_{m})\geq c\cdot k^{-m/2}\cdot p_{1}^{-k_{1}}\cdot\dots\cdot p_{m}^{-k_{m}}.
\]
\end{lem}

\begin{proof}
Choose $(i_{1},\dots,i_{k})\in\mathcal{I}^{k}$ at random such that
$P(i_{j}=l)=p_{l}$ independently for each $j$. Then
\begin{equation}
P\left(\left|\left\{ j:1\leq j\leq k,i_{j}=l\right\} \right|=k_{l}\,\mathrm{for\,every}\,1\leq l\leq m\right)=N(k_{1},\dots,k_{m})p_{1}^{k_{1}}\cdot\dots\cdot p_{m}^{k_{m}}.\label{eq:p_i product}
\end{equation}
Let
\[
N_{k,l}=N_{k,l}(\boldsymbol{\mathbf{i}})=\left|\left\{ j:1\leq j\leq k,i_{j}=l\right\} \right|=\sum_{j=1}^{k}\mathbf{1}_{i_{j}=l}
\]
for $\boldsymbol{\mathbf{i}}=(i_{1},\dots,i_{k})\in\mathcal{I}^{k}$.
We have that $\mathrm{E}(N_{k,l})=kp_{l}$ and $\mathrm{E}\left((N_{k,l}-kp_{l})^{2}\right)=kp_{l}(1-p_{l})$,
hence by Chebyshev's inequality
\[
P\left(\left|N_{k,l}-kp_{l}\right|\geq\sqrt{2k}\right)\leq\frac{kp_{l}(1-p_{l})}{2k}\leq p_{l}/2.
\]
Thus
\[
P\left(\left|N_{k,l}-kp_{l}\right|<\sqrt{2k}\,\mathrm{for\,every}\,1\leq l\leq m\right)\geq1-\sum_{l=1}^{m}p_{l}/2=1/2.
\]
Then by (\ref{eq:p_i product})
\[
\sum_{\begin{array}{c}
kp_{l}-\sqrt{2k}<k_{l}<kp_{l}+\sqrt{2k}\\
\mathrm{for\,every\,}l=1,\dots,m
\end{array}}N(k_{1},\dots,k_{m})p_{1}^{k_{1}}\cdot\dots\cdot p_{m}^{k_{m}}
\]
\[
=P\left(\left|N_{k,l}-kp_{l}\right|<\sqrt{2k}\,\mathrm{for\,every}\,1\leq l\leq m\right)\geq1/2.
\]
There are less than $\left(2\sqrt{2k}+1\right)^{m}\leq\left(4\sqrt{k}\right)^{m}$
terms in the sum, so there exist $k_{1},\dots,k_{m}\in\mathbb{N}$
such that $\sum_{l=1}^{m}k_{l}=k$
\[
N(k_{1},\dots,k_{m})p_{1}^{k_{1}}\cdot\dots\cdot p_{m}^{k_{m}}\geq2^{-1}\left(4\sqrt{k}\right)^{-m}=2^{-2m-1}\cdot k^{-m/2}
\]
which completes the proof.
\end{proof}

\begin{note}
Let $(k_{1},\dots,k_{m})\in\mathbb{N}^{m}$ be the closest (or one
of the closest) point to $(kp_{1},\dots,kp_{m})\in\mathbb{R}^{m}$
such that $(k_{1},\dots,k_{m})$ is on the hyperplane $x_{1}+\dots+x_{m}=k$.
Using Stirling`s formula one can show that
\[
N(k_{1},\dots,k_{m})\geq c\cdot k^{\left(1-m\right)/2}\cdot p_{1}^{-k_{1}}\cdot\dots\cdot p_{m}^{-k_{m}}
\]
for some $c>0$ independent of $k$. However, the conclusion of Lemma
\ref{lem:Cheb lem} is enough for us so we have given a more elementary
proof for that.
\end{note}

The main idea of the proof of Theorem \ref{thm:approx trafo thm}
is based on the proof by Peres and Shmerkin \cite[Proposition 6]{Peres-Schmerkin Resonance between Cantor sets}.\\

\selectlanguage{english}%
\noindent \textit{Proof of Theorem }\foreignlanguage{british}{\ref{thm:approx trafo thm}.
From Theorem \ref{thm: MAIN} there exists an SS-IFS $\left\{ \widehat{S_{i}}\right\} _{i=1}^{m}$
that satisfies the SSC with attractor $\widehat{K}$ such that $\widehat{K}\subseteq K_{j}$,
$\dim_{H}K_{j}-\frac{\varepsilon}{2}<\dim_{H}\widehat{K}$ and the
transformation group $\mathcal{T}$ of $\left\{ \widehat{S_{i}}\right\} _{i=1}^{m}$
is dense in $\mathcal{T}_{j,G}$. Let $s>0$ be the unique solution
of $\sum_{i=1}^{m}r_{i}^{s}=1$ where $r_{i}$ are the similarity
ratios of the maps $\widehat{S_{i}}$. Since $\left\{ \widehat{S_{i}}\right\} _{i=1}^{m}$
satisfies the SSC it follows that $\dim_{H}\widehat{K}=s>\dim_{H}K_{j}-\frac{\varepsilon}{2}$,
see (\ref{eq: S-dim sum}).}

\selectlanguage{british}%
Since $\overline{\mathcal{T}}$ is compact we can take a finite open
$\frac{\varepsilon}{2}$-cover $\left\{ U_{i}\right\} _{i=1}^{n}$
of $\overline{\mathcal{T}}$ and for every $i\in\left\{ 1,\dots,n\right\} $
we take $O_{i}\in U_{i}\cap\mathcal{T}$. For every $i\in\left\{ 1,\dots,n\right\} $
we fix a finite word $\boldsymbol{\mathbf{j}}_{i}\in\bigcup_{k=1}^{\infty}\mathcal{I}^{k}$
such that $\left\Vert T_{\boldsymbol{\mathbf{j}}_{i}}-O_{i}^{-1}\circ O\right\Vert <\frac{\varepsilon}{2}$
(we can find such a $\boldsymbol{\mathbf{j}}_{i}$ due to the compactness
of $\overline{\mathcal{T}}$, see for example \cite[Lemma 2.1]{Farkas-Linim-only}).
Then for every $T\in U_{i}$ it follows that
\begin{multline}
\left\Vert T\circ T_{\boldsymbol{\mathbf{j}}_{i}}-O\right\Vert \leq\left\Vert T\circ T_{\boldsymbol{\mathbf{j}}_{i}}-O_{i}\circ T_{\boldsymbol{\mathbf{j}}_{i}}+O_{i}\circ T_{\boldsymbol{\mathbf{j}}_{i}}-O\right\Vert \leq\left\Vert T-O_{i}\right\Vert +\left\Vert T_{\boldsymbol{\mathbf{j}}_{i}}-O_{i}^{-1}\circ O\right\Vert <\varepsilon.\label{eq:right trafo eq}
\end{multline}

Let $k\in\mathbb{N}$, $p_{l}=r_{l}^{s}$ for every $1\leq l\leq m$.
Then by Lemma \ref{lem:Cheb lem} there exist $k_{1},\dots,k_{m}\in\mathbb{N}$
such that $\sum_{l=1}^{m}k_{l}=k$ and
\[
N(k_{1},\dots,k_{m})\geq c\cdot k^{\left(1-m\right)/2}\cdot p_{1}^{-k_{1}}\cdot\dots\cdot p_{m}^{-k_{m}}
\]
for some $c>0$ independent of $k$. Then for every
\[
\boldsymbol{\mathbf{i}}\in\mathcal{J}_{0}:=\left\{ (i_{1},\dots,i_{k})\in\mathcal{I}^{k}:\left|\left\{ j:1\leq j\leq k,i_{j}=l\right\} \right|=k_{l}\,\mathrm{for\,every}\,1\leq l\leq m\right\} 
\]
it follows that 
\[
\rho:=r_{\boldsymbol{\mathbf{i}}}=\prod_{l=1}^{m}r_{l}^{k_{l}}
\]
and
\[
\left|\mathcal{J}_{0}\right|=N(k_{1},\dots,k_{m})\geq c\cdot k^{-m/2}\prod_{l=1}^{m}r_{i}^{-sk_{l}}.
\]

Since $\left\{ U_{i}\right\} _{i=1}^{n}$ is a finite $\frac{\varepsilon}{2}$-cover
of $\overline{\mathcal{T}}$ we can find $U_{i}$ such that for at
least $n^{-1}\left|\mathcal{J}_{0}\right|$ words $\boldsymbol{\mathbf{i}}\in\mathcal{J}_{0}$
we have that $T_{\boldsymbol{\mathbf{i}}}\in U_{i}$. Let $\mathcal{J}=\left\{ \boldsymbol{\mathbf{i}}\in\mathcal{J}_{0}:T_{\boldsymbol{\mathbf{i}}}\in U_{i}\right\} $,
then $\left\Vert T_{\boldsymbol{\mathbf{i}}}-O_{i}\right\Vert <\frac{\varepsilon}{2}$,
$r_{\boldsymbol{\mathbf{i}}}=\rho$ for every $\boldsymbol{\mathbf{i}}\in\mathcal{J}$
and
\[
\left|\mathcal{J}\right|\geq n^{-1}N(k_{1},\dots,k_{m})\geq n^{-1}c(\sqrt{k})^{-m}\prod_{l=1}^{m}r_{i}^{-sk_{l}}.
\]
Let $K$ be the attractor of the SS-IFS $\left\{ \widehat{S}_{\boldsymbol{\mathbf{i}}}\circ\widehat{S}_{\boldsymbol{\mathbf{j}}_{i}}:\boldsymbol{\mathbf{i}}\in\mathcal{J}\right\} $.
Then $r:=r_{\boldsymbol{\mathbf{i}}}r_{\boldsymbol{\mathbf{j}}_{i}}=\rho r_{\boldsymbol{\mathbf{j}}_{i}}$,
$\left\Vert T_{\boldsymbol{\mathbf{i}}}\circ T_{\boldsymbol{\mathbf{j}}_{i}}-O\right\Vert <\varepsilon$
by (\ref{eq:right trafo eq}) for every $\boldsymbol{\mathbf{i}}\in\mathcal{J}$
and by (\ref{eq: S-dim sum})
\[
\begin{alignedat}{1}\dim_{H}K=\frac{\log\left|\mathcal{J}\right|}{-\log\rho r_{\boldsymbol{\mathbf{j}}_{i}}} & \geq\frac{-\log n+\log c-m\log(\sqrt{k})-s\left(\sum_{l=1}^{m}k_{l}\log r_{l}\right)}{-\left(\sum_{l=1}^{m}k_{l}\log r_{l}\right)-\log r_{\boldsymbol{\mathbf{j}}_{i}}}\\
 & \geq s-\frac{\varepsilon}{2}>\dim_{H}K_{j}-\varepsilon
\end{alignedat}
\]
if $k$ is large enough.\foreignlanguage{english}{$\hfill\square$}\\

\begin{lem}
\label{lem:r_1 lem}Let $\left\{ S_{i}\right\} _{i=1}^{m}$ be an
SS-IFS in $\mathbb{R}^{d}$ that satisfies the SSC with attractor
$K$ and let $\varepsilon>0$. Let $\widehat{K}$ be the attractor
of the SS-IFS $\left\{ S_{\boldsymbol{\mathbf{i}}}\circ f_{\boldsymbol{\mathbf{i}}}:\boldsymbol{\mathbf{i}}\in\mathcal{I}^{k}\right\} $
where either $f_{\boldsymbol{\mathbf{i}}}=S_{1}$ or $f_{\boldsymbol{\mathbf{i}}}=Id_{\mathbb{R}^{d}}$.
For $k$ large enough $\dim_{H}K-\varepsilon<\dim_{H}\widehat{K}\leq\dim_{H}K$.
\end{lem}

\begin{proof}
It follows that $\dim_{H}\widehat{K}\leq\dim_{H}K$ because $\widehat{K}\subseteq K$.
Since the SSC is satisfied $\dim_{H}K=s$ where $\sum_{i=1}^{m}r_{i}^{s}=1$,
see (\ref{eq: S-dim sum}). Thus $\sum_{i=1}^{m}r_{i}^{s-\varepsilon}>1$.
Let $k\in\mathbb{N}$ be such that $\left(\sum_{i=1}^{m}r_{i}^{s-\varepsilon}\right)^{k}>1/r_{1}^{s-\varepsilon}$.
Let $p_{\boldsymbol{\mathbf{i}}}=r_{1}$ if $f_{\boldsymbol{\mathbf{i}}}=S_{1}$
and $p_{\boldsymbol{\mathbf{i}}}=1$ if $f_{\boldsymbol{\mathbf{i}}}=Id_{\mathbb{R}^{d}}$.
Then
\[
\sum_{\boldsymbol{\mathbf{i}}\in\mathcal{I}^{k}}r_{\boldsymbol{\mathbf{i}}}^{s-\varepsilon}p_{\boldsymbol{\mathbf{i}}}^{s-\varepsilon}\geq\sum_{\boldsymbol{\mathbf{i}}\in\mathcal{I}^{k}}r_{\boldsymbol{\mathbf{i}}}^{s-\varepsilon}r_{1}^{s-\varepsilon}=\left(\sum_{i=1}^{m}r_{i}^{s-\varepsilon}\right)^{k}r_{1}^{s-\varepsilon}>1
\]
and hence $\dim_{H}K-\varepsilon=s-\varepsilon<\dim_{H}\widehat{K}$
because $\left\{ S_{\boldsymbol{\mathbf{i}}}\circ f_{\boldsymbol{\mathbf{i}}}:\boldsymbol{\mathbf{i}}\in\mathcal{I}^{k}\right\} $
satisfies the SSC.
\end{proof}

\selectlanguage{english}%
\noindent \textit{Proof of Theorem }\foreignlanguage{british}{\ref{thm:planar dense thm}.
}If $\left|\mathcal{T}_{j,G}\right|<\infty$ then Theorem \ref{thm:planar dense thm}
follows from Corollary \ref{cor:finite approx cor} because $\mathcal{T}_{j,G}\cap\mathbb{SO}_{2}$
is a finite cyclic group since $\mathbb{SO}_{2}$ is commutative,
so we assume that $\left|\mathcal{T}_{j,G}\right|=\infty$. By \foreignlanguage{british}{Theorem
\ref{thm: MAIN} there exists an SS-IFS $\left\{ \widehat{S_{i}}\right\} _{i=1}^{m}$
that satisfies the SSC with attractor $\widehat{K}$ such that $\widehat{K}\subseteq K_{j}$,
$\dim_{H}K_{j}-\frac{\varepsilon}{2}<\dim_{H}\widehat{K}$ and the
transformation group $\mathcal{T}$ of $\left\{ \widehat{S_{i}}\right\} _{i=1}^{m}$
is dense in $\mathcal{T}_{j,G}$. It is easy to see that $\left|\mathcal{T}\right|=\infty$
implies that $\mathcal{T}$ contains a rotation of infinite order.
If $\mathcal{T}$ contains reflections, without loss of generality
say $T_{1}$ is a reflection, we iterate the SS-IFS $\left\{ \widehat{S_{i}}\right\} _{i=1}^{m}$
a large number of times and compose the orientation reversing maps
with $\widehat{S_{1}}$. The new SS-IFS looks like $\left\{ \widehat{S}_{\boldsymbol{\mathbf{i}}}\circ f_{\boldsymbol{\mathbf{i}}}:\boldsymbol{\mathbf{i}}\in\mathcal{I}^{k}\right\} $
where $f_{\boldsymbol{\mathbf{i}}}=\widehat{S_{1}}$ if $T_{\boldsymbol{\mathbf{i}}}\notin\mathbb{SO}_{2}$
and $f_{\boldsymbol{\mathbf{i}}}=Id_{\mathbb{R}^{2}}$ if $T_{\boldsymbol{\mathbf{i}}}\in\mathbb{SO}_{2}$.
Since $\mathcal{T}$ contains a rotation of infinite order it follows
that the transformation group of $\left\{ \widehat{S}_{\boldsymbol{\mathbf{i}}}\circ f_{\boldsymbol{\mathbf{i}}}:\boldsymbol{\mathbf{i}}\in\mathcal{I}^{k}\right\} $
contains a rotation of infinite order. It follows from Lemma \ref{lem:r_1 lem}
that if we choose $k$ large enough then the Hausdorff dimension of
the attractor of $\left\{ \widehat{S}_{\boldsymbol{\mathbf{i}}}\circ f_{\boldsymbol{\mathbf{i}}}:\boldsymbol{\mathbf{i}}\in\mathcal{I}^{k}\right\} $
approximates $\dim_{H}\widehat{K}$. Hence we can assume that there
exists an SS-IFS $\left\{ \widehat{S_{i}}\right\} _{i=1}^{m}$ that
satisfies the SSC with attractor $\widehat{K}$ such that $\widehat{K}\subseteq K_{j}$,
$\dim_{H}K_{j}-\frac{\varepsilon}{2}<\dim_{H}\widehat{K}$ and $\mathcal{T}\subseteq\mathbb{SO}_{2}$
contains a rotation of infinite order.}

The rest of the proof is very similar to the proof of Theorem \ref{thm:approx trafo thm}\foreignlanguage{british}{.
There is no need for the $\frac{\varepsilon}{2}$-cover $\left\{ U_{l}\right\} _{l=1}^{p}$
of $\overline{\mathcal{T}}$. Instead we fix $\boldsymbol{\mathbf{j}}\in\mathcal{I}^{l}$
for some $l\in\mathbb{N}$ such that $T_{\boldsymbol{\mathbf{j}}}$
is of infinite order. From here on we proceed as in the the proof
of Theorem }\ref{thm:approx trafo thm}\foreignlanguage{british}{
with minor differences. Since $\mathbb{SO}_{2}$ is commutative it
follows that for all the $N(k_{1},\dots,k_{m})$ words $\boldsymbol{\mathbf{i}}\in\mathcal{J}_{0}$
we have that $T_{\boldsymbol{\mathbf{i}}}=T$ for some $T\in\mathcal{T}$.
Then either $T$ or $T\circ T_{\boldsymbol{\mathbf{j}}}$ is of infinite
order. Hence proceeding as in the proof of Theorem }\ref{thm:approx trafo thm}\foreignlanguage{british}{
we can show that either $\left\{ \widehat{S}_{\boldsymbol{\mathbf{i}}}:\boldsymbol{\mathbf{i}}\in\mathcal{J}\right\} $
or $\left\{ \widehat{S}_{\boldsymbol{\mathbf{i}}}\circ\widehat{S}_{\boldsymbol{\mathbf{j}}_{l}}:\boldsymbol{\mathbf{i}}\in\mathcal{J}\right\} $
satisfies the theorem.$\hfill\square$}\\

\selectlanguage{british}%

\selectlanguage{english}%
\begin{center}
$\mathbf{Acknowledgements}$
\par\end{center}

\selectlanguage{british}%

\selectlanguage{english}%
The author was supported by an EPSRC doctoral training grant. The
author thanks Kenneth Falconer for many valuable suggestions on improving
the exposition of the article and for the simplification of Lemma
\ref{lem:Cheb lem}.

\selectlanguage{british}%

\end{document}